\documentclass[reqno,11pt]{amsart}

\usepackage[utf8]{inputenc}
\usepackage[english]{babel}
\usepackage{amsfonts}
\usepackage{amssymb}
\usepackage{amsmath}
\usepackage{color}
\usepackage{mathrsfs}
\usepackage{accents}
\usepackage{esint}
\usepackage{anyfontsize}
\usepackage{enumitem}

\usepackage[fixamsbugs]{mathtools}

\usepackage[colorlinks, citecolor=citegreen, linkcolor=refred,urlcolor=blue, unicode,psdextra,hypertexnames=false]{hyperref}
\definecolor{citegreen}{rgb}{0,0.3,0}
\definecolor{refred}{rgb}{0.5,0,0}

\usepackage{cleveref}
\usepackage{autonum}

\theoremstyle{plain}
\newtheorem{theorem}{Theorem}[section]
\newtheorem{lemma}[theorem]{Lemma}

\newtheorem{ackn}{Acknowledgements\hspace{-.2em}}

\theoremstyle{definition}
\newtheorem{definition}[theorem]{Definition}

\theoremstyle{remark}
\newtheorem{remark}[theorem]{Remark}

\crefname{definition}{Definition}{Definitions}
\crefname{theorem}{Theorem}{Theorems}
\crefname{lemma}{Lemma}{Lemmas}
\crefname{step}{Step}{Steps}
\crefname{substep}{Step}{Steps}
\crefname{claim}{Claim}{Claims}
\crefname{proposition}{Proposition}{Propositions}
\crefname{corollary}{Corollary}{Corollaries}
\crefname{remark}{Remark}{Remarks}
\crefname{section}{Section}{Sections}
\crefname{subsection}{Section}{Sections}
\crefname{chapter}{Chapter}{Chapters}
\crefname{figure}{Figure}{Figures}
\crefname{appendix}{Appendix}{Appendices}
\crefname{equation}{}{}
\crefname{enumi}{}{}
\crefname{enumii}{}{}

\numberwithin{equation}{section}

\def\ringg#1{\accentset{\circ}{#1}}
\def\varempty{\text{\rm{\O}}}
 
\def\emptyarg{\kern1pt\cdot\kern1pt}
\newcommand{\R}{\mathbb{R}}
\newcommand{\N}{\mathbb{N}}

\newcommand{\dif}{\,\mathrm{d}}

\newcommand{\D}{\mathrm{D}\kern.01cm}

\DeclareMathOperator{\ncapa}{\mathfrak{c}}

\let\oldchi\chi
\renewcommand{\chi}{\raisebox{\depth}{\(\oldchi\)}}

\DeclareMathOperator{\ee}{e}

\newcommand{\MF}{\mathscr{F}}
\newcommand{\MG}{\mathscr{G}}

\newcommand{\kst}[1][C]{\mathrm{#1}}

\let\H\relax\DeclareMathOperator{\H}{H}
\DeclareMathOperator{\h}{h}

\DeclareMathOperator{\Ric}{Ric}
\newcommand{\AVR}{\mathrm{AVR}}
\newcommand{\sca}{\mathrm{R}}

\makeatletter
\def\abs#1{\ifx\relax#1\relax\empty|\emptyarg|\else
 {\mathchoice%
 {\@trimleftright{\lvert}{\rvert}{\displaystyle}{#1}}%
 {\lvert #1\rvert}%
 {\@trimleftright{\lvert}{\rvert}{\scriptstyle}{#1}}%
 {\@trimleftright{\lvert}{\rvert}{\scriptscriptstyle}{#1}}}%
 \fi%
}
\makeatother

\AtBeginDocument{%
   \def\MR#1{}
}

\renewenvironment{cases}[1][@{\ }rcl@{\hspace{1cm}}l]{\bgroup\arraycolsep=1.4pt\left\lbrace\kern-3pt\begin{array}{#1}}{\end{array}\right.\egroup}

\makeatletter
\newcommand*{\@trimleftright}[4]{%
 \sbox0{$#3#4\m@th$}%
 \sbox2{$#3\vcenter{}$}
 \sbox4{$#3\left#1\right#2$}%
 \dimen0=\dimexpr\ht0\relax %
 \ifdim\dimen0<\ht2 %
 \dimen0=\ht2 %
 \fi
 \dimen2=\dp0 %
 \ifdim\dimen2<\z@
 \dimen2=\z@
 \fi
 \ifdim\ht4>\dimen0 %
 \dimen0=\ht4 %
 \fi
 \ifdim\dp4>\dimen2 %
 \dimen2=\dimexpr\dp4\relax %
 \fi
 \dimen0=\dimexpr(\dimen0-\dimen2)/2 -\ht2\relax
 \raisebox{\dimen0}{%
 $#3\left#1\raisebox{-\dimen0}{\box0}\right#2\m@th$%
 }%
}

\makeatother

\makeatletter
\patchcmd{\@setaddresses}{\indent}{\noindent}{}{}
\patchcmd{\@setaddresses}{\indent}{\noindent}{}{}
\patchcmd{\@setaddresses}{\indent}{\noindent}{}{}
\patchcmd{\@setaddresses}{\indent}{\noindent}{}{}
\makeatother

\title[A Note on Ricci--pinched three--manifolds]{A Note on Ricci--pinched three--manifolds}

\date{\today}
 
\author[Luca Benatti]{Luca Benatti}
\address[Luca Benatti]{Dipartimento di Matematica, Universit\`a di Pisa, Italy}
\email[L. Benatti]{luca.benatti@unipi.it}

\author[Carlo Mantegazza]{Carlo Mantegazza}
\address[Carlo Mantegazza]{Dipartimento di Matematica e Applicazioni "Renato Caccioppoli", Universit\`a di Napoli Federico II \& Scuola Superiore Meridionale, Napoli, Italy}
\email[C. Mantegazza]{carlo.mantegazza@unina.it}

\author[Francesca Oronzio]{Francesca Oronzio}
\address[Francesca Oronzio]{Institutionen f\"{o}r Matematik, Kungliga Tekniska H\"{o}gskolan, Stockholm, Sweden}
\email[F. Oronzio]{oronzio@kth.se}

\author[Alessandra Pluda]{Alessandra Pluda}
\address[Alessandra Pluda]{Dipartimento di Matematica, Universit\`a di Pisa, Italy}
\email[A. Pluda]{alessandra.pluda@unipi.it}

\date{\today}
\subjclass[2020]{53C21, 53E10, 83C99, 49J45}

\keywords{Riemannian $3$--manifold, Ricci--pinching, linear potential theory}

\usepackage{mathptmx}
\usepackage{palatino}
\usepackage{a4wide}

\begin{document}

\begin{abstract}
Let $(M, g)$ be a complete, connected, non--compact Riemannian $3$--manifold. Suppose that $(M,g)$ satisfies the {\em Ricci--pinching condition} 
$\Ric\geqslant\varepsilon\mathrm{R} g$ for some $\varepsilon>0$, where $\Ric$ and $\mathrm{R}$ are the Ricci tensor and scalar curvature, respectively. In this short note, we give an alternative proof based on potential theory of the fact that if $(M,g)$ has Euclidean volume growth, then it is flat. Deruelle--Schulze--Simon~\cite{deruelle_schulze_simon_2022} and  Huisken--K\"{o}rber~\cite{huisken_koerber_24} have already shown this result and together with the contributions by Lott~\cite{lott_2024} and Lee--Topping~\cite{lee_topping_2022} led to a proof of the so--called {\em Hamilton's pinching conjecture}.
\end{abstract}

\maketitle


\section{Introduction}

Let $(M, g)$ be a complete and connected Riemannian $3$--manifold. We denote by $\Ric$ and $\mathrm{R}$ the Ricci and scalar curvature, respectively.
\begin{definition}
A Riemannian manifold $(M,g)$ is {\em Ricci--pinched} if $\Ric\geqslant 0$ and there exists a constant $\varepsilon>0$ such that $\Ric\geqslant\varepsilon\mathrm{R} g$.
\end{definition}

The following theorem was known as {\em Hamilton's pinching conjecture} 
and its proof required the joint efforts of Lott~\cite{lott_2024}, Deruelle--Schulze--Simon~\cite{deruelle_schulze_simon_2022} and Lee--Topping~\cite{lee_topping_2022}.

\begin{theorem}\label{thmpinchingtheorem}
Let $(M,g)$ be a complete, connected Riemannian $3$--manifold. Suppose that $(M,g)$ is Ricci--pinched, then it is  flat or compact.
\end{theorem}
Notice that being flat or compact is not mutually exclusive, consider for instance a flat $3$--torus. This result is a generalization of the well--known {\em Myers's diameter estimate}: if $(M,g)$ is a complete and connected $n$--dimensional Riemannian manifold such that $\Ric\geqslant (n-1)k^2g$, for some constant $k>0$, then $M$ is compact and $\mathrm{diam}(M,g)\leqslant\pi/k$. Richard~Hamilton conjectured \cref{thmpinchingtheorem}, possibly taking inspiration from its extrinsic counterpart that he proved for hypersurfaces of the Euclidean space~\cite{hamilton_1994}.

\begin{theorem}
Let $M$ be a smooth, strictly convex, complete hypersurface  in $\R^{n}$. 
If the second fundamental form of $M$ is {\em pinched}, in the sense that there exists $\varepsilon>0$ such that
\begin{equation}
h_{ij}\geqslant\varepsilon\H g_{ij},
\end{equation}
where $ g_{ij}$ is the induced Riemannian metric, then $M$ is compact. 
\end{theorem}

A first step towards the proof of \cref{thmpinchingtheorem} was done by Chen and Zhu~\cite{chenzhu} who proved, employing the Ricci flow, that a $3$--dimensional, complete and non--compact Riemannian manifold, with bounded and nonnegative sectional curvature, which is Ricci--pinched is flat. Then, Lott~\cite{lott_2024} improved their result, requiring milder assumptions on the sectional curvature, and Deruelle--Schulze--Simon~\cite{deruelle_schulze_simon_2022} showed that the conjecture is true if the curvature is bounded.
Finally, Lee and Topping~\cite{lee_topping_2022} removed the bounded curvature assumption. All these results employ the Ricci flow. We mention that higher--dimensional versions of Hamilton’s conjecture were proven by Ma and Cheng in~\cite{Ma_Cheng} and by Deruelle--Schulze--Simon~\cite{deruelle_schulze_simon_2024} (see also~\cite{chan_lee_peachey_2024}).

In this short note, we give an alternative, direct, and mostly self--contained proof of a weaker version of \cref{thmpinchingtheorem}.

The {\em asymptotic volume ratio} of $(M,g)$ is defined as 
\begin{equation}\label{avreq}
\AVR=\frac{3}{4\pi}\lim_{r\to +\infty}\frac{\mathrm{Vol}(B_r(p))}{r^3},
\end{equation}
for any point $p\in M$. When $\Ric\geqslant0$, thanks to the Bishop--Gromov theorem, the quantity $\AVR$ is well--defined and independent of the point $p\in M$. Moreover, $\AVR\in [0,1]$ and $\AVR=1$ if and only if the manifold is $\R^3$ endowed with the Euclidean metric.

\begin{theorem}\label{thm:main}
Let $(M, g)$ be a complete, connected, non--compact, Ricci--pinched Riemannian $3$--manifold. Suppose that $\AVR>0$, then $(M,g)$ is flat.
\end{theorem}

We will get \cref{thm:main} as a consequence of a slightly more general result, where the assumption $\AVR>0$ is replaced by another condition on the asymptotic volume growth. We say that $(M,g)$ has {\em superquadratic volume growth} if there exist a point $p\in M$ and two constants $\kst_{\mathrm{vol}}> 0$ and $\alpha\in (1,2]$ such that, for sufficiently large $r$, 
\begin{equation}\label{superquadratic}
\kst_{\mathrm{vol}}^{-1}\, r^{1+\alpha}\leqslant{\mathrm{Vol}}(B_r(q))\leqslant \kst_{\mathrm{vol}}\, r^{1+\alpha}.
\end{equation}

\begin{theorem}\label{thm:main-2}
Let $(M, g)$ be a complete, connected, non--compact, Ricci--pinched Riemannian $3$--manifold. Suppose that $(M,g)$ has {\em superquadratic volume growth} with $\alpha>4/3$ in \cref{superquadratic}, then $(M,g)$ is flat.
\end{theorem}

Condition~\eqref{superquadratic} holding with $\alpha=2$ is equivalent to $\AVR>0$, hence \cref{thm:main} is a special case of \cref{thm:main-2}. We mention that  \cref{thm:main-2} is contained in the paper of Deruelle--Schulze--Simon~\cite[Theorem 1.3]{deruelle_schulze_simon_2022} and has been proved also by Huisken--K\"{o}erber using the {\em inverse mean curvature flow}~\cite{huisken_koerber_24}. Our proof in the next sections avoids the existence and regularity theory for the inverse mean curvature flow~\cite{heidusch_phdthesis_2001,huisken_inversemeancurvatureflow_2001}, being replaced with the more widely known potential theory. At the end of the paper, we also show an application to manifolds with boundary.

\section{Proof of \texorpdfstring{\cref{thm:main-2}}{}}

Let $(M, g)$ be a complete, connected, non--compact, orientable, Ricci--pinched Riemannian $3$--manifold. We suppose by contradiction that $(M,g)$ is not flat, then there must exist a point $o\in M$ with $\sca(o)>0$. As a consequence, by considering the asymptotic expansion of the surface area and the mean curvature $\H$ of the small spheres $\partial B_{r}(o)$, as $r\to 0$, there exists a radius $r\ll 1$ such that $\partial B_{r}(o)$ is a smooth surface and
\begin{equation}\label{F(0)}
\int_{\partial B_{r}(o)}\H^2\dif\mu<16\pi,
\end{equation} 
see for instance~\cite[Theorem 3.2]{fan_shi_tam_2009}.\\
We then set $\Omega= \overline{B}_{r}(o)$ and we define the function $w$ as the solution of the elliptic problem
\begin{equation}\label{eq:moser_2_potential}
\begin{cases}
\Delta w& =&\abs{\nabla w }^2 &\text{on $ M\setminus\Omega $}\\
w&=&0&\text{on $\partial\Omega$}\\
w&\to&+\infty &\text{as $ d(x,o)\to+\infty$}
\end{cases}
\end{equation}
The existence and regularity of such a solution are granted by the classical theory of harmonic functions. Consider indeed the following problem
\begin{equation}\label{eq:harmonic_potential}
\begin{cases}
\Delta u& =& 0 &\text{on $ M\setminus\Omega $}\\
u&=&1&\text{on $\partial\Omega$}\\
u&\to&0 &\text{as $ d(x,o)\to+\infty$}
\end{cases}
\end{equation}
and assume that $(M,g)$ has superquadratic volume growth, that is condition~\eqref{superquadratic} holds. Then, if $\Omega\subseteq M$ is a regular domain, problem~\cref{eq:harmonic_potential} admits a unique solution $u\in C^\infty(M\setminus\ringg\Omega)$ which takes values in $(0,1]$ and it is smooth till the boundary (see the papers by Varopoulos~\cite{Varopoulos}, Li--Yau~\cite{Li-Yau86} and Agostiniani--Fogagnolo--Mazzieri~\cite{agostiniani_sharpgeometricinequalitiesclosed_2020}). Then, $w=-\log u$ is a smooth solution of problem~\eqref{eq:moser_2_potential}.

Let $\Omega_t=\{w\leqslant t\}\cup \Omega$. We define the following function
$\MF$ at every regular value $t\in[0,+\infty)$ of $w$ solution of problem~\eqref{eq:moser_2_potential}, as
$$
\MF(t)= \int_{\partial\Omega_t}\H\vert\nabla w\vert -\vert\nabla w\vert^2\dif\mu\label{eq:F},
$$
where $\H$ denotes the mean curvature with respect to the outward pointing unit normal $\nu=\nabla w/|\nabla w|$ and $\mu$ is the surface measure of the level set $\partial\Omega_t=\{w=t\}$. By Sard theorem the set of critical values of $w$ has zero Lebesgue measure, hence the function $\MF$ is then well defined almost everywhere in $[0,+\infty)$.\\
Notice that, by simply expanding the square in $\left(\H\!/2-\vert\nabla w\vert\right)^2\geqslant 0$, we have
\begin{equation}\label{FHdis}
\MF(t)=\int_{\partial\Omega_t}\H\vert\nabla w\vert -\vert\nabla w\vert^2\dif\mu\leqslant\int_{\partial\Omega_t}\H^2\!\!/4\,\dif\mu.
\end{equation}
In particular, being $\partial\Omega_0=\partial B_r(o)$ a regular level set of $w$, we have $\MF(0)<4\pi$, by equation~\eqref{F(0)}.

The following lemma is in the spirit of similar results in~\cite{agostiniani_sharpgeometricinequalitiesclosed_2020, AMO24}.

\begin{lemma}\label{monotonicity}
The function $\MF$ admits a locally absolutely continuous, nonincreasing extension (still denoted by $\MF$) to the whole $[0,+\infty)$. Moreover, 
at the regular values of $w$, there holds
\begin{equation}\label{explicit-derivative}
\MF'(t)=\,-\,\int_{\partial\Omega_t}\Bigg[
\frac{\abs{\nabla^\top\abs{\nabla w}}^2}{\abs{\nabla w}^2} +\Ric(\nu,\nu)+\vert{\ringg{\h}}\vert^2 +\frac{1}{2}\left({\H} -2\abs{\nabla w}\right)^2 \Bigg]\,\dif\mu\,\leqslant\, 0,
\end{equation}
where $\nu=\nabla w/|\nabla w|$ and $\h$ are the outward pointing unit normal and the second fundamental form of $\partial\Omega_t$, $\ringg{\h}$ the traceless part of $\h$ and $\nabla^\top$ denotes the tangential part of the gradient (with respect to $\partial\Omega_t$).
\end{lemma}
\begin{proof}
At every regular value $t\in[0,+\infty)$ of $w$, it is straightforward to see that
\begin{equation}
\MF(t)=-\int_{\partial\Omega_t}\bigg\langle\nabla\abs{\nabla w}\,,\frac{\nabla w}{\abs{\nabla w}}\bigg\rangle\,\dif\mu,\qquad\text{hence}\qquad \MF(t)-\MF(s)=-\int_{\{s<w<t\}}\mathrm{div}\,(\nabla\abs{\nabla w})\,\dif\mu,
\end{equation}
(by the divergence theorem) for every pair of regular values $s<t$ of $w$ in $[0,+\infty)$ such that the open set $\{s<w<t\}$ has no critical points.\\
The vector field $\nabla\abs{\nabla w}$ is well defined and smooth outside the set of the critical points of $w$ and by direct computation, we get
$$
\mathrm{div}\,(\nabla\abs{\nabla w})=|\nabla w|\,\Bigg[
\frac{\abs{\nabla^\top\abs{\nabla w}}^2}{\abs{\nabla w}^2} +\Ric(\nu,\nu)+\vert{\ringg{\h}}\vert^2 +\frac{1}{2}\left({\H} -2\abs{\nabla w}\right)^2 \Bigg].
$$
If the open set $\{s<w<t\}$ does not contain critical points of $w$, then the inequality $\MF(s)-\MF(t)\geqslant 0$ follows and equation~\eqref{explicit-derivative} is immediate. If instead the open set $\{s<w<t\}$ contains some critical points, to obtain the same conclusion, one can use appropriate approximating vector fields $\eta(|\nabla w|)\nabla\abs{\nabla w}$, smooth on all $M\setminus\Omega$ and with nonnegative divergence, as in~\cite{agostiniani_sharpgeometricinequalitiesclosed_2020, AMO24}. Following such argument, one also gets that $\MF\in W^{1,1}_{\mathrm{loc}}(0,+\infty)$, with a weak derivative given almost everywhere by formula~\eqref{explicit-derivative}.
\end{proof}

\begin{lemma}\label{lem:limit}
There exists $\widetilde{t}\in [0,+\infty)$ such that for all $t\geqslant\widetilde{t}$, there holds $\MF(t)\leqslant\kst\ee^{-2t}$, for a positive constant $\kst$.
\end{lemma}
\begin{proof}
If $\Sigma$ is a closed, connected surface in $(M,g)$ with $\Ric\geqslant\varepsilon\mathrm{R} g$, we have
\begin{align} 
2\int_{\Sigma}\Ric(\nu,\nu)\,\dif\mu
 &\geqslant\varepsilon\left(16\pi-\int_{\Sigma}\H^2\dif\mu\right) &\quad\text{if}\;\mathrm{genus}(\Sigma)=0,\label{eq:genere_zero}\\
2\int_{\Sigma}\Ric(\nu,\nu)+\abs{\ringg{\h}}^2\dif\mu&\geqslant\int_{\Sigma}\H^2\dif\mu&\quad\text{if}\;\mathrm{genus}(\Sigma)\geqslant 1.\label{eq:genere_maggiore_uno}
\end{align}
These two inequalities follow from the Gauss--Bonnet theorem and the Gauss--Codazzi equations, taking into account the pinching condition in the first one (see~\cite[Lemma 8]{huisken_koerber_24}).\\
Suppose that $t\geqslant0$ is a regular value of $w$, then the number of the connected components of $\partial \Omega_t$ is finite, by its compactness. If all of them have genus greater or equal to one, by inequality~\eqref{explicit-derivative} and using estimate~\eqref{eq:genere_maggiore_uno} for every single connected component, after adding we obtain
\begin{equation}
-2\MF'(t)\geqslant \int_{\partial\Omega_t}
2\Ric(\nu,\nu)+2\vert{\ringg{\h}}\vert^2\dif\mu \geqslant \int_{\partial\Omega_t}\H^2\dif\mu\geqslant 4\MF(t),
\end{equation}
where the last inequality is given by formula~\eqref{FHdis}. If there exists at least one connected component with genus zero, letting $\Sigma^1_t\neq \varempty$ be the union of the $n\in\N$ connected components of genus zero and $\Sigma^2_t$ the union of the connected components of genus greater than one, by inequalities~\eqref{explicit-derivative} and~\eqref{eq:genere_zero}, we have
\begin{align}
-2\MF'(t)&\geqslant \int_{\partial\Omega_t} 2\Ric(\nu,\nu) +\left({\H} -2\abs{\nabla w}\right)^2\dif\mu\\
&\geqslant\int_{\Sigma^1_t} 2\Ric(\nu,\nu) +\varepsilon\left({\H} -2\abs{\nabla w}\right)^2\dif\mu+\varepsilon\int_{\Sigma^2_t} \left({\H} -2\abs{\nabla w}\right)^2\dif\mu\\
&\geqslant \varepsilon\left(16n\pi -4\int_{\Sigma^1_t} \H\abs{\nabla w}-\abs{\nabla w}^2\dif\mu\right)
-4\varepsilon\int_{\Sigma^2_t} \H\abs{\nabla w}-\abs{\nabla w}^2\dif\mu\\
&=\varepsilon\left(16n\pi -4\int_{\partial\Omega_t} \H\abs{\nabla w}-\abs{\nabla w}^2\dif\mu\right)\\
&\geqslant\varepsilon\left(16\pi -4\MF(t)\right),
\end{align}
where we used the fact that $\varepsilon\leqslant 1/3$ (this follows by tracing the Ricci--pinching condition).\\
Hence, we can conclude that for almost every $t\in [0,+\infty)$, there holds 
\begin{equation}
\MF'(t)\leqslant\max\{-2\MF(t),\varepsilon\left(2\MF(t)-8\pi\right)\}.
\end{equation}
The thesis then follows from this differential inequality, keeping into account that $\MF$ is locally absolutely continuous, by \cref{monotonicity}. Indeed, by the monotonicity of $\MF$, either $\MF(t) \geqslant 8\pi\varepsilon/(2+2\varepsilon)$ for every $t\geqslant 0$, or there exists $\widetilde{t}\geqslant 0$ such that $\MF(t) \leqslant 8\pi\varepsilon/(2+2\varepsilon)$ for every $t \geqslant \widetilde{t}$. In the first case, $\MF'(t)\leqslant \varepsilon(2 \MF(t)- 8\pi) $, for every $t\geqslant 0$ and $\MF(t) \leqslant \MF(0) < 4 \pi$. Hence, there must exist some $t\geqslant 0$ such that $\MF(t) < 8\pi\varepsilon/(2+2\varepsilon)$, which is a contradiction. In the second case, $\MF'(t)\leqslant-2\MF(t)$ for all $t\geqslant\widetilde{t}$, which implies $\MF(t)\leqslant 4\pi \ee^{-2(t- \widetilde{t})}$, hence the thesis.
\end{proof}

Now we introduce another function $\MG$, defined at every regular value $t\in[0,+\infty)$ of $w$ as
$$
\MG(t)= \int_{\partial\Omega_t}\vert\nabla w\vert^2\dif\mu.
$$

\begin{lemma}\label{F-positive}
For almost every $t\in [0,+\infty)$, there holds $0\leqslant \MG(t)\leqslant \MF(t)$. In particular, 
$$
\lim_{t\to+\infty}\MF(t)=\lim_{t\to+\infty}\MG(t)=0.
$$
\end{lemma}
\begin{proof}
As a consequence of~\cite[Theorem 3.1]{benatti_fogagnolo_mazzieri_minkowskii_2022} the function $\MG$ admits a nonincreasing $C^1$--extension on all $[0,+\infty)$ (indeed, $\MG(t)=F_2^1(e^{t})$, where $F^\beta_{p}$ are the monotone quantities introduced in~\cite{benatti_fogagnolo_mazzieri_minkowskii_2022}). One can readily check that at every regular value $t\in[0,+\infty)$ of $w$ (almost all, by Sard theorem), we have
$$0\geqslant\MG'(t)=\MG(t)-\MF(t),$$
which gives the thesis.
\end{proof}

We then need the notion of {\em normalized capacity} of a bounded closed set $D\subseteq M$:
\begin{equation}
\ncapa_2(\partial D) =\inf\bigg\{\frac{1}{4\pi }\int_{M\setminus D}\abs{\nabla\psi}^2\dif\mathrm{Vol}\,\,\bigg\vert\,\,\psi\in C^\infty_c(M),\psi\geqslant\chi_D\bigg\}.
\end{equation}
The relation of such capacity with the function $w$ is given by the fact that (recalling that $w=-\log u$ with $u$ the harmonic function solving problem~\eqref{eq:harmonic_potential}) 
\begin{equation}
\ncapa_2(\partial\Omega)=\frac{1}{4\pi }\int_{M\setminus\Omega}\abs{\nabla u}^2\dif\mathrm{Vol}=\frac{1}{4\pi}\int_{\partial\Omega}\abs{\nabla u}\,\dif\mu=\frac{1}{4\pi}\int_{\partial\Omega}\abs{\nabla w}\,\dif\mu,
\end{equation}
where we kept into account that $|\nabla w|=|\nabla u|$ on $\partial\Omega$, as $u=1$ (see~\cite[Proposition~2.8]{benatti_fogagnolo_mazzieri_minkowskii_2022} for a detailed justification of the first two equalities). Moreover, with the same argument, at every regular value $t\in[0,+\infty)$ of $w$, we have (\cite[Proposition~2.9]{benatti_fogagnolo_mazzieri_minkowskii_2022})
\begin{equation}\label{eqcar1}
\ncapa_2(\partial\Omega_t)
=\frac{1}{4\pi}\int_{\partial\Omega_t}\abs{\nabla w}\,\dif\mu=\frac{e^t}{4\pi}\int_{\partial\Omega_t} \abs{\nabla u}\,\dif\mu=\frac{e^t}{4\pi}\int_{\partial\Omega} \abs{\nabla u}\,\dif\mu=\ee^t\ncapa_2(\partial\Omega) ,
\end{equation}
where we used again the divergence theorem in the domain $\Omega_t\setminus\Omega$.

\medskip

\begin{proof}[Proof of \cref{thm:main-2}] 
We consider first the case in which $(M,g)$ is orientable.

We need the following ``classical'' estimates for a solution $u:M\setminus \Omega \to (0,1]$ of problem~\cref{eq:harmonic_potential} (see for instance~\cite{agostiniani_sharpgeometricinequalitiesclosed_2020, Cheng-Yau75,Li-Yau86}):
there exist a positive constant $\kst=\kst(M,\Omega)$ such that for all $x\in M\setminus \Omega$, 
\begin{equation}\label{Li-Yau}
u(x) \leqslant \kst\, d(x,o)^{1-\alpha},
\end{equation}
where $\alpha$ is the exponent in condition~\eqref{superquadratic}.\\
By equation~\eqref{eqcar1} and H\"older inequality, at every regular value $t\in [0,+\infty)$ of $w$, we have
\begin{equation}
\ee^{3t} \ncapa_2(\partial \Omega)^3
=\ncapa_2(\partial \Omega_t)^3=\left(\frac{1}{4\pi} \int_{\partial \Omega_t} \abs{\nabla w}\,\dif\mu \right)^3 
\leqslant \frac{1}{(4\pi)^3}\left( \int_{\partial \Omega_t} \abs{\nabla w}^{-1}\,\dif\mu\right) \left( \int_{\partial \Omega_t} \abs{\nabla w}^2\dif\mu\right)^2
\end{equation}
and from Lemmas~\ref{lem:limit} and~\ref{F-positive}, we know that there exists $\widetilde{t}\in [0,+\infty)$ such that for all $t\in [\widetilde{t}, +\infty)$, there holds
\begin{equation}
\int_{\partial \Omega_t} \abs{\nabla w}^2\dif\mu=\MG(t)\leqslant\kst \ee^{-2t},
\end{equation}
for a positive constant $\kst$. Thus, using the coarea formula, we obtain
\begin{equation}
\frac{d}{dt}{\mathrm{Vol}}(\{w\leqslant t\})=\int_{\partial\Omega_t}\abs{\nabla w}^{-1} \,\dif\mu\geqslant\big[\, 4\pi \ncapa_2(\partial \Omega)\big]^3 \ee^{3t}/\MG^2(t)\geqslant\big[\, 4\pi \ncapa_2(\partial \Omega)\big]^3 \ee^{7t}/\kst^2.
\end{equation}
for almost every $t\in [0,+\infty)$. Let $R_t=\sup\big\lbrace d(q,o)\,:\, q\in \{w\leqslant t\}=\Omega_t\big\rbrace$ for any $t\in[0,+\infty)$ and $t_n\to+\infty$ be an increasing sequence of regular values of $w$ (whose existence is again guaranteed by Sard theorem). Integrating the above inequality on $[0, t_n]$ and using the superquadratic volume growth assumption, we get
\begin{equation}\label{stima}
\frac{1}{7\kst^{2}} \big[\,4\pi \ncapa_2(\partial \Omega)\big]^3\left(\ee^{7t_n}-1\right)
\leqslant{\mathrm{Vol}}(\{w\leqslant t_n\})\leqslant {\mathrm{Vol}}(B_{R_{t_n}}(o))\leqslant \kst_{\mathrm{vol}} \,R_{t_n}^{1+\alpha}.
\end{equation}
Being $w=-\log u$, by estimate~\eqref{Li-Yau}, we have $w(x) \geqslant -\log(\kst d(x,o)^{1-\alpha})$, then if $d(q,o)=R_{t_n}$, it must be $q\in\partial\Omega_{t_n}$, that is, $w(q)=t_n$ and we have
$$
t_n=w(q)\geqslant -\log(\kst d(q,o)^{1-\alpha})=-\log(\kst R_{t_n}^{1-\alpha}),
$$
hence, $R_{t_n}^{\alpha-1}\leqslant \kst \ee^{t_n}$, which implies $R_{t_n}^{\alpha+1}\leqslant \kst \ee^{\frac{\alpha+1}{\alpha-1}t_n}$ for a positive constant $\kst=\kst(M,\Omega)$. Then, by inequality~\eqref{stima}, we conclude that
\begin{equation}
\ee^{7t_n}-1
\leqslant \kst R_{t_n}^{\alpha+1}\leqslant \kst e^{\frac{\alpha+1}{\alpha-1}t_n},
\end{equation}
which is clearly a contradiction if $\alpha>{4}/{3}$, as $t_n$ can be chosen arbitrarily large.  This argument proves the thesis under the additional assumption that the manifold is orientable. In particular, $(M,g)$ is isometric to the Euclidean space $\mathbb{R}^3$, which is the only flat Riemannian $3$-manifold with superquadratic volume growth with $\alpha>\frac{4}{3}$ in \eqref{superquadratic}.

Now, suppose that $(M,g)$ is nonorientable. We consider its orientable double cover, which is also Ricci-pinched and has superquadratic volume growth. Hence, this cover must be Euclidean space. This yields a contradiction, as $\R^3$ cannot be the double cover of any nonorientable Riemannian $3$-manifold.
\end{proof}

Replacing the ``starting subset" $\overline{B}_r(o)$ with a different regular subset $\Omega$ with a compact boundary, such that
\begin{equation}\label{F(0)bis}
\int_{\partial\Omega}\H^2\dif\mu<16\pi,
\end{equation}
and repeating the above argument, one obtains the same conclusion. It is then straightforward to obtain also the following result when $M$ has a boundary.

\begin{theorem}\label{thm:main-3}
There not exist a complete, connected, non--compact, orientable Ricci--pinched Riemannian $3$--manifold $(M, g)$ that has superquadratic volume growth with $\alpha>4/3$ in \cref{superquadratic} and a compact smooth boundary $\partial M$ satisfying
\begin{equation}\label{F(0)bis2}
\int_{\partial M}\H^2\dif\mu<16\pi.
\end{equation}
\end{theorem}

\begin{remark}
Clearly, as before, the case $\AVR>0$ correspond to the case $\alpha=2$ in assumption~\eqref{superquadratic}, hence, in particular, 
there not exist a complete, connected, non--compact, orientable, Ricci--pinched Riemannian $3$--manifold $(M, g)$ with $\AVR>0$ and a compact smooth boundary $\partial M$ satisfying condition~\eqref{F(0)bis2}.
\end{remark}

\begin{remark}
If one is interested in proving only \cref{thm:main-2}, it is known that $M$ must be diffeomorphic to $\R^3$ (see~\cite{zu_shun-hui_1994}), then $M$ is orientable and thanks to the strong maximum principle, one can show that the regular level sets of $w$ are connected (see~\cite{AMMO, AMO24} for more detail). This observation simplifies a little bit the proof of \cref{lem:limit} in such a case.
\end{remark}

\begin{ackn}
All the authors are members of the INDAM--GNAMPA. Luca Benatti and Alessandra Pluda are partially supported by the BIHO Project ``NEWS -- NEtWorks and Surfaces evolving by curvature'' and by the MUR Excellence Department Project awarded to the Department of Mathematics of the University of Pisa. Carlo Mantegazza is partially supported by the PRIN Project 2022E9CF89 ``GEPSO -- Geometric Evolution Problems and Shape Optimization''. Luca Benatti is partially supported by PRIN Project 2022PJ9EFL ``Geometric Measure Theory: Structure of Singular Measures, Regularity Theory and Applications in the Calculus of Variations''. Alessandra Pluda is partially supported by the PRIN Project 2022R537CS ``$\rm{NO}^3$ -- Nodal Optimization, NOnlinear elliptic equations, NOnlocal geometric problems, with a focus on regularity''.
\end{ackn}

\bibliographystyle{amsplain}
\bibliography{Ricci-pinched}

\end{document}